\documentclass[a4paper,twoside]{amsart}
\usepackage{xypic,amscd,amssymb,latexsym,srcltx}
\usepackage{dsfont}
\usepackage{amsthm}
\def\@setcopyright{}
\makeatother

\newtheorem{lemma}{Lemma}[section]

\newtheorem{theorem}[lemma]{Theorem}
\newtheorem{remark}[lemma]{Remark}
\newtheorem*{acknowledgments*}{ACKNOWLEDGMENTS}
\newtheorem{definition}[lemma]{Definition}

\newtheorem*{thm*}{Main Theorem}
\newtheorem*{remark*}{Remark}
\newtheorem*{lemma*}{Lemma}
\newtheorem*{example*}{[Example]}

\begin{document}
\begin{center}
{\Large \bf Fusion systems of blocks with nontrivial strongly closed subgroups
}\\

\bigskip
\end{center}

\begin{center}
Heguo Liu$^{1}$, Xingzhong Xu$^{1,2}$*, Jiping Zhang$^{3}$

\end{center}
\footnotetext {
1.Department of Mathematics, Hubei University,
Wuhan 430062, China

2. Departament de Matem$\mathrm{\grave{a}}$tiques, Universitat Aut$\mathrm{\grave{o}}$noma de Barcelona, E-08193 Bellaterra,
Spain

3. School of Mathematical Sciences, Peking University, Beijing, 100871, China

Heguo Liu's E-mail: ghliu@hubu.edu.cn

Xingzhong Xu's E-mail: xuxingzhong407@mat.uab.cat, xuxingzhong407@126.com

$\ast$ Corresponding author

Jiping Zhang's E-mail: jzhang@pku.edu.cn

}

\title{}
\def\abstractname{\textbf{Abstract}}

\begin{abstract}\addcontentsline{toc}{section}{\bf{English Abstract}}
In this paper, we find some exotic fusion systems which have nontrivial strongly closed subgroups,
and we prove these fusion systems are also not realizable by $p$-blocks of finite groups.
\hfil\break


\end{abstract}

\maketitle

\section{\bf Introduction}

To get the $p$-local information of finite group $G$,
one can forget some information of this group and change the structure of a
group to a category. The objects of this category are all subgroups of a given Sylow
$p$-subgroup $S$ of $G$,
and the morphism set of $Q, P(\leq S)$ is the set $\mathrm{Hom}_{G}(Q, P):=\{c_{g}|Q^{g}\leq P\}$.
This category is a saturated fusion system $\mathcal{F}_{S}(G)$ of finite group $G$ over $S$.

Alperin's theorem for fusion systems tells us that every morphism  in $\mathcal{F}_{S}(G)$ can be generated by
some automorphisms of objects in $\mathcal{F}^{fcr}$(see theorem 2.4).
So we can get all the $p$-local information of $G$ from the automorphisms of objects in $\mathcal{F}^{fcr}$.
The more general concept of a
" saturated fusion system $\mathcal{F}$ over a finite $p$-group $S$ " was introduced by Lluis Puig (See
\cite{BLO, Pu}).
Our main result is the following:

\begin{thm*}  Let $p$ be an odd prime and $P=\langle x,y,z|x^{p}=y^{p}=z^{p}=1, [x,y]=z,
[z,x]=[z,y]=1\rangle\cong p^{1+2}_{+}$. Let $\mathcal{F}_{1}$ and $ \mathcal{F}_{2}$ be a saturated fusion system over $P$ and $A$ respectively. Here,
$A$ is a finite abelian $p$-group. Let $\mathcal{F}= \mathcal{F}_{1}\times \mathcal{F}_{2}$.
 Then
$\mathcal{F}$ is  not a fusion system of a $p$-block of a finite group if and only if
$\mathcal{F}_{1}$ is  not a fusion system of a $p$-block of a finite group.
\end{thm*}

This theorem is motivated by Kessar's problem about fusion systems of blocks: Does the exotic fusion system is realizable by
a fusion system of a block of a finite group? In \cite{Ke}, she had proven that the Solomon fusion system $\mathcal{F}_{\mathrm{Sol}}(3)$
is not a fusion system of a $2$-block of a finite group. In \cite{KeS}, they used their reduction theorem to prove that
the exotic fusion systems over $p_{+}^{1+2}$ do not occur as fusion systems of $p$-blocks of finite groups.

These exotic fusion systems which are studied in \cite{Ke, KeS} have no nontrivial strongly closed subgroups(see \cite[Proposition 4.3(iv)]{Ke},\cite[Lemma 4.17]{RV}).
In this paper, we fine the exotic fusion systems which have nontrivial strongly closed subgroups, and in Theorem 3.4, we
prove these fusion systems are also not realizable by $p$-blocks of finite groups. In fact, our result relies on
methods in \cite{Ke, KeS} and theory of fusion systems of blocks built in \cite{AsKO}.

The paper is organized as follows.
We start in Section 2 with the basic definitions on fusion systems, Alperin's theorem for fusion systems
and factor fusion systems.
In Section 3, we prove the main theorem (see Theorem 3.4).

\section{\bf Preliminaries}

In this section we collect some known results that will be needed later.  For the background theory of
fusion systems, we refer to \cite{AsKO}, \cite{RS}, \cite{St2}.

\begin{definition} A $fusion ~ system$ $\mathcal{F}$ over a finite $p$-group
$S$ is a category whose objects are the subgroups of $S$, and whose
morphism sets $\mathrm{Hom}_{\mathcal{F}}(P,Q)$ satisfy the
following two conditions:

\vskip 0.3cm

(a) $\mathrm{Hom}_{S}(P,Q)\subseteq
\mathrm{Hom}_{\mathcal{F}}(P,Q)\subseteq\mathrm{Inj}(P,Q)$ for all
$P,Q\leq S$.

\vskip 0.3cm

(b) Every morphism in $\mathcal{F}$ factors as an isomorphism in
$\mathcal{F}$ followed by an inclusion.

\end{definition}

\begin{definition} Let $\mathcal{F}$ be a fusion system over a $p$-group
$S$.

$\bullet$ Two subgroups $P,Q$ are $\mathcal{F}$-$conjugate$ if they
are isomorphic as objects of the category $\mathcal{F}$. Let
$P^{\mathcal{F}}$ denote the set of all subgroups of $S$ which are
$\mathcal{F}$-conjugate to $P$.
Since $\mathrm{Hom}_{\mathcal{F}}(P,P)\subseteq\mathrm{Inj}(P,P)$,
we usually write $\mathrm{Hom}_{\mathcal{F}}(P,P)=\mathrm{Aut}_{\mathcal{F}}(P)$ and
$\mathrm{Hom}_{S}(P,P)=\mathrm{Aut}_{S}(P)$.

$\bullet$ A subgroup $P\leq S$ is $fully~automised$ in $\mathcal{F}$
if $\mathrm{Aut}_{S}(P)\in
\mathrm{Syl}_{p}(\mathrm{Aut}_{\mathcal{F}}(P))$.

$\bullet$ A subgroup $P\leq S$ is $receptive$ in $\mathcal{F}$ if it
has the following property: for each $Q\leq S$ and each $\varphi\in
\mathrm{Iso}_{\mathcal{F}}(Q, P)$, if we set
$$N_{\varphi}=\{g\in N_{S}(Q)|\varphi \circ c_{g}\circ \varphi^{-1}\in \mathrm{Aut}_{S}(P)\},$$
then there is $\overline{\varphi}\in
\mathrm{Hom}_{\mathcal{F}}(N_{\varphi},S)$ such that
$\overline{\varphi}|_{Q}=\varphi$. (where $c_{g}:x\longmapsto g^{-1}xg$ for $g\in S$)

$\bullet$ A fusion system $\mathcal{F}$ over a $p$-group $S$ is
$saturated$ if each subgroup of $S$ is $\mathcal{F}$-conjugate to a
subgroup which is fully automised and receptive. This definition is due to \cite{RS}.
\end{definition}

\begin{definition} Let $\mathcal{F}$ be a fusion system over a $p$-group
$S$.

$\bullet$ A subgroup $P\leq S$ is $fully~centralised$ in
$\mathcal{F}$ if $|C_{S}(P)|\geq |C_{S}(Q)|$ for all $Q\in
P^{\mathcal{F}}$.

$\bullet$ A subgroup $P\leq S$ is $fully~normalised$ in
$\mathcal{F}$ if $|N_{S}(P)|\geq |N_{S}(Q)|$ for all $Q\in
P^{\mathcal{F}}$.

$\bullet$ A subgroup $P\leq S$ is $\mathcal{F}$-$centric$ if
$C_{S}(Q)=Z(Q)$ for $Q\in P^{\mathcal{F}}$.

$\bullet$ A subgroup $P\leq S$ is $\mathcal{F}$-$radical$ if
$\mathrm{Out}_{\mathcal{F}}(P)=\mathrm{Aut}_{\mathcal{F}}(P)/\mathrm{Inn}(P)$
is $p$-reduced; i.e., if $O_{p}(\mathrm{Out}_{\mathcal{F}}(P))=1$.

$\bullet$ Let $\mathcal{F}^{r}, \mathcal{F}^{c}, \mathcal{F}^{f}$ denote the full subcategories whose objects are
$\mathcal{F}$-radical, $\mathcal{F}$-centric, and fully~normalised in $\mathcal{F}$ respectively.
Let $\mathcal{F}^{fcr}=\mathcal{F}^{f}\cap\mathcal{F}^{c}\cap\mathcal{F}^{r}$
and $\mathcal{F}^{cr}=\mathcal{F}^{c}\cap\mathcal{F}^{r}$.

$\bullet$ A subgroup $P\leq S$ is $normal$ in $\mathcal{F}$ (denoted
$P\trianglelefteq \mathcal{F}$) if for all $Q,R\in S$ and all
$\varphi\in \mathrm{Hom }_{\mathcal{F}}(Q,R)$, $\varphi$ extends to
a morphism $\overline{\varphi}\in \mathrm{Hom
}_{\mathcal{F}}(QP,RP)$ such that $(P)\overline{\varphi}=P$.

$\bullet$ A subgroup $P\leq S$ is $strongly~closed$ in $\mathcal{F}$
 if no element of $P$ is $\mathcal{F}$-conjugate to an element of
 $S-P$.

\end{definition}

\begin{definition} A saturated fusion system $\mathcal{F}$ over a $p$-group $S$
will be called realizable if $\mathcal{F}=\mathcal{F}_{S}(G)$ for some finite group $G$
with $S\in \mathrm{Syl}_{p}(G)$, and will be called exotic otherwise.
\end{definition}

\begin{theorem}(Alperin's theorem for fusion systems). Let $\mathcal{F}$ be a saturated fusion system over a $p$-group
$S$. Then for each morphism $\varphi\in \mathrm{Iso}_{\mathcal{F}}(P,R)$ in
$\mathcal{F}$, there exist sequences of subgroups of $S$
$$P=P_{0},P_{1},\ldots,P_{k}=R,~~~~~ \mathrm{and }~~~~~ Q_{1},Q_{2},\ldots,Q_{k},$$
and morphism $\psi_{i}\in \mathrm{Aut}_{\mathcal{F}}(Q_{i})$, such
that

$(a)$ $Q_{i}$ is fully~normalised in $\mathcal{F}$,
$\mathcal{F}$-radical, and $\mathcal{F}$-centric for each $i$(that is $Q_{i}\in \mathcal{F}^{fcr}$);

$(b)$ $P_{i-1},P_{i}\leq Q_{i}$ and $\psi_{i}(P_{i-1})=P_{i}$ for
each $i$; and

$(c)$ $\psi=\psi_{k}\circ\psi_{k-1}\circ\cdots\circ\psi_{1}$.

\end{theorem}

The following is devoted to recalling the basic definitions of
factor fusion systems, we refer to \cite{A5},\cite{AsKO}.

\begin{definition} Let $\mathcal{F}$ be a saturated fusion system over a $p$-group
$S$, and fix $Q\leq S$ and $K\leq \mathrm{Aut}(Q)$. Let $N_{\mathcal{F}}^{K}(Q)\subseteq \mathcal{F}$
be the fusion system over $N_{S}^{K}(Q)$ where $P, R \leq N_{S}^{K}(Q)$,
$$\mathrm{Hom}_{N_{\mathcal{F}}^{K}(Q)}(P, R)=\{\varphi\in \mathrm{Hom}_{\mathcal{F}}(P, R)|\exists~ \overline{\varphi}\in \varphi\in \mathrm{Hom}_{\mathcal{F}}(PQ, RQ) $$ $$ ~~~ ~~~~~~~~~~~~~~~~~~~~~~~~~~~~~~~~~~~~~~~~~~~~~~~
with~ \overline{\varphi}|_{P}=\varphi, \overline{\varphi}(Q)=Q, ~and~ \overline{\varphi}|_{Q}\in K\}.$$
As special cases, set $N_{\mathcal{F}}(Q)=N_{\mathcal{F}}^{\mathrm{Aut}(Q)}(Q)$ and $C_{\mathcal{F}}(Q)=N_{\mathcal{F}}^{\{1\}}(Q)$.
If $Q$ is fully $K$-normalized in $\mathcal{F}$, then $N_{\mathcal{F}}^{K}(Q)$ is a saturated fusion system.
\end{definition}

\begin{definition} Let $\mathcal{F}$ be a saturated fusion system over a $p$-group
$S$, $T$ a strongly closed in $S$ with respect to $\mathcal{F}$,
and $\mathcal{N}=N_{\mathcal{F}}(T)$.

$\bullet$ Set $S^{+}=S/T$ and $\theta: S\longrightarrow
S^{+}$ be the natural map $\theta: x\longrightarrow x^{+}=xT\in S^{+}$.

$\bullet$ We define a category $\mathcal{F}/T$: $\mathcal{F}/T$ and a morphism
$\theta: \mathcal{N}\longrightarrow
\mathcal{F}/T$. The objects of $\mathcal{F}/T$ are the subgroups of $S^{+}$. For $P\leq S$ and
$\alpha\in \mathrm{Hom}_{\mathcal{N}}(P,S)$ define $\alpha^{+}:
P^{+}\longrightarrow S^{+}$ by $\alpha^{+}(x^{+})=\alpha(x)^{+}$. This is well defined as
$T$ a strongly closed in $S$ with respect to $\mathcal{F}$. Now define
$$\mathrm{Hom}_{\mathcal{F}/T}(P^{+},S^{+})=\{\beta^{+}|\beta\in \mathrm{Hom}_{\mathcal{N}}(PT,S)\}.$$
and define $\theta_{P}: \mathrm{Hom}_{\mathcal{N}}(P, S)\longrightarrow
\mathrm{Hom}_{\mathcal{F}/T}(P^{+}, S^{+})$ by $\theta_{P}(\alpha)=\alpha^{+}$. For
$\alpha\in \mathrm{Hom}_{\mathcal{N}}(P,S)$, $\alpha$ extends to
$\hat{\alpha}\in \mathrm{Hom}_{\mathcal{N}}(PT,S)$ and $\hat{\alpha}^{+}=\alpha^{+}$,
so $\theta_{P}$ is well defined and surjective.
\end{definition}

We next look at the definition of products of fusion systems(see \cite{A6, BLO}).

\begin{definition}Let $\mathcal{F}_{1}$ and $ \mathcal{F}_{2}$ be a saturated fusion system over finite $p$-groups
 $S_{1}$ and $S_{2}$ respectively.
$\mathcal{F}_{1}\times \mathcal{F}_{2}$ is fusion system over $S_{1}\times S_{2}$ generated by the set
of all $(\varphi_{1},\varphi_{2})\in \mathrm{Hom}(P_{1}\times P_{2}, Q_{1}\times Q_{2})$ for $\varphi_{i}\in
\mathrm{Hom}_{\mathcal{F}}(P_{i}, Q_{i})$.
\end{definition}

\begin{remark}Let $\mathcal{F}_{1}$ and $ \mathcal{F}_{2}$ be a saturated fusion system over $S_{1}$ and $S_{2}$  respectively.
Let $\mathcal{F}= \mathcal{F}_{1}\times \mathcal{F}_{2}$.
 Then $\mathcal{F}/S_{1}\cong C_{\mathcal{F}}(S_{1})/S_{1} \cong \mathcal{F}_{2}$.
\end{remark}

\section{\bf Proof of the Main theorem}

For sake of completeness, we recall the theory of fusion systems of blocks(see \cite[Part IV]{AsKO}).

Let $k$ be an algebraically closed field of characteristic $p$, $G$ a finite group and $b$ a block of $kG$, that is $b$ is a primitive idempotent of
$Z(kG)$. For any $p$-subgroup $Q$ of $G$ the canonical projection from $kN$ to $kC_{G}(Q)$
induces an algebra morphism $\mathrm{Br}_{Q}^{kG}$ from the subalgebra of fixed points of
$Q$, $(kG)^{Q}$ onto $kC_{G}(Q)$(see \cite{AsKO, BP}). This morphism is known in the literature as
Brauer morphism.

A $(kG, b, G)$-Brauer pair is a pair $(Q, e_{Q})$ where $Q$ is a $p$-subgroup of $G$ such that
$\mathrm{Br}_{Q}^{kG}(b)\neq 0$ and $e_{Q}$ is a block of $kC_{G}(Q)$ such that
$\mathrm{Br}_{Q}^{kG}(b)e_{Q}\neq 0 $. Let $(Q, e_{Q}), (R, e_{R})$ be $(kG, b, G)$-Brauer pair.
We say that $(Q, e_{Q})$ is contained in $(R, e_{R})$ if $Q\leq R$ and
for any primitive idempotent $i\in (kG)^{R}$ such that
$\mathrm{Br}_{Q}^{kG}(i)e_{R}\neq 0$, we have $\mathrm{Br}_{Q}^{kG}(i)e_{Q}\neq 0$.
Here, we write $(Q, e_{Q})\leq (R, e_{R})$. Recall \cite{AsKO, BP}, we know that

(1) Given a $(kG, b, G)$-Brauer pair $(R, e_{R})$ and $Q\leq R$, there exists a unique
$(c, G)$-Brauer pair  $(Q, e_{Q})$ contained in $(R, e_{R})$.

(2) All maximal $(kG, b, G)$-Brauer pair are $G$-conjugate.

\begin{definition}Let $Q\leq P$ be $p$-subgroups of $G$.

$\bullet$ Define
$$\mathrm{Br}_{P,Q}^{kG}: \mathrm{Br}_{Q}^{kG}((kG)^{P})\rightarrow kC_{G}(P)$$
by $\mathrm{Br}_{P,Q}^{kG}(\mathrm{Br}_{Q}^{kG}(a))=\mathrm{Br}_{P}^{kG}(a)$ for $a\in (kG)^{P}$.
If $e_{Q}$ is a block of $kC_{G}(Q)$ and $e_{Q}$ is $P$-stable, then
$\mathrm{Br}_{P,Q}^{kG}(\mathrm{Br}_{Q}^{kG}(e_{Q}))=\mathrm{Br}_{P,Q}^{kG}(e_{Q})=\mathrm{Br}_{P}^{kG}(e_{Q})=e_{Q}$.

$\bullet$ Let $(Q, e_{Q}), (P, e_{P})$ be $(kG, b, G)$-Brauer pair. We write that
$(Q, e_{Q})\unlhd (P, e_{P})$ if $Q$ is a normal subgroup of $P$, $e_{Q}$ is $P$-stable and
$\mathrm{Br}_{P,Q}^{kG}(\mathrm{Br}_{Q}^{kG}(e_{Q})) e_{P}=e_{P}$. That is $e_{Q}e_{P}=e_{P}.$

$\bullet$ For a  $(kG, b, G)$-Brauer pair $(Q, e_{Q})$, let us denote by $N_{G}(Q, e_{Q})$the stabilizer
in $N_{G}(Q)$ of $e_{Q}$.

\end{definition}

For further
standard notation and terminology, we follow \cite{AsKO}, \cite{NT}.

\begin{definition} Let $(S, e_{S})$ is the maximal
$(kG, b, G)$-Brauer pair. For each $Q\leq S$, let $e_{Q}$ denote
the unique block of $kC_{G}(Q)$ such that $(Q, e_{Q})\leq (S, e_{S})$.
The fusion category of $(kG, b, G)$ over $S, e_{S}$ is the category
$\mathcal{F}_{(S, e_{S})}(kG, b, G)$ whose objects are the subgroups of
$S$ and which has morphism sets
$$\mathrm{Mor}_{\mathcal{F}_{(S, e_{S})}(kG, b, G)}(Q, R)=\mathrm{Hom}_{G}((Q, e_{Q}),(R, e_{R}))$$
for $Q, R\leq S$ $\mathrm{(}$see \cite[Part IV, Defintion 2.21]{AsKO}$\mathrm{)}$.
\end{definition}

\begin{definition}If an abstract fusion system $\mathcal{F}\cong\mathcal{F}_{(S, e_{S})}(kG, b, G)$, we say that fusion system $\mathcal{F}$
is block realizable.  By Brauer's Third Main Theorem, any realizable saturated
fusion system is block realizable. Moreover, if $\mathcal{F}$ is not block realizable, then
$\mathcal{F}$ is exotic. $\mathrm{(}$see \cite[Part IV, p. 221]{AsKO}$\mathrm{)}$.
\end{definition}

\begin{theorem}  Let $p$ be an odd prime and $P=\langle x,y,z|x^{p}=y^{p}=z^{p}=1, [x,y]=z,
[z,x]=[z,y]=1\rangle\cong p^{1+2}_{+}$. Let $\mathcal{F}_{1}, \mathcal{F}_{2}$ be a saturated fusion system over $P$ and $A$ respectively. Here,
$A$ is a finite abelian $p$-group.
Let $\mathcal{F}= \mathcal{F}_{1}\times \mathcal{F}_{2}$. Then
$\mathcal{F}$ is  not a fusion system of a $p$-block of a finite group if and only if
$\mathcal{F}_{1}$ is  not a fusion system of a $p$-block of a finite group.
\end{theorem}

\begin{proof}[\textbf{Proof}] Let $S:=P\times A$. Suppose that $\mathcal{F}$ is a fusion system of a $p$-block of a finite group,
we can set $\mathcal{F}=\mathcal{F}_{(S, e_{S})}(kG, b, G)$ where $G$ is a finite group, $b$ is a block of $G$, and
$(S, e_{S})$ is a maximal $(kG, b, G)$-Brauer pair.

Since $\mathcal{F}= \mathcal{F}_{1}\times \mathcal{F}_{2}$, we have $A$ is strongly closed in $\mathcal{F}$.
Let $(A, e_{A})$ be a $(kG, b, G)$-Brauer pair such that $(A, e_{A})\leq (S, e_{S})$.
By \cite[Part IV, Proposition 3.19(b)]{AsKO}, we have
$$C_{\mathcal{F}}(A)=N_{\mathcal{F}}^{K}(A)=\mathcal{F}_{(S, e_{S})}(kC_{G}(A), e_{A}, C_{G}(A))$$
by \cite[Part IV, Proposition 3.8]{AsKO}. Here, $K=\mathrm{Aut}_{C_{G}(A)}(A)=\{\mathrm{Id}_{A}\}$.

By Definition 2.7, we have $C_{\mathcal{F}}(A)/A\cong \mathcal{F}/A\cong \mathcal{F}_{1}$. Since $A$ is abelian, we have $A\leq C_{G}(A)$.
As before, we denote by $\mu_{A}$ the natural $k$-algebra homomorphism
$$\mu_{A}: kC_{G}(A)\longrightarrow k(C_{G}(A)/A)$$
$$~~~~\sum_{x}\lambda_{x}x\longmapsto \sum_{x}\lambda_{x}\overline{x},$$
where $\overline{x}=xA\in C_{G}(A)/A$.
As $e_{A}$ is a block of $kC_{G}(A)$ and $|C_{G}(A):C_{C_{G}(A)}(A)|=1$, then $\overline{e}_{A}:=\mu(e_{A})$ is a block of $k(C_{G}(A)/A)$
by \cite[Ch. 5, Theorem 8.11]{NT}.

In the following, we will check that
$$C_{\mathcal{F}}(A)/A\subseteq\mathcal{F}_{(S/A, \overline{\varepsilon}_{S})}(kC_{G}(A)/A, \overline{e}_{A},C_{G}(A)/A).$$
Here, $\overline{\varepsilon}_{S}$ will be defined in the following text.

First, we will prove that there is an $C_{G}(A)$-poset equivalence between the subset of
$(kC_{G}(A), e_{A}, C_{G}(A))$-Brauer pairs consisting of those pairs who contains $(A, e_{A})$,
and the subset of
$(kC_{G}(A)/A, \overline{e}_{A}, C_{G}(A)/A)$-Brauer pairs by the map $\Phi$
$$(T, e_{T})\longmapsto \Phi(T, e_{T}).$$
Here, $(T, e_{T})$ is a $(kC_{G}(A), e_{A}, C_{G}(A))$-Brauer pair such that $(T, e_{T})\leq (S, e_{S})$. We will define the $\Phi(T, e_{T})$ in the following.

Since $A\unlhd T\leq S$, set $\overline{T}=T/A$.
It is easy to see that $C_{G}(T)/A\leq C_{C_{G}(A)/A}(\overline{T})$, we can set
that $C_{C_{G}(A)/A}(\overline{T})=L_{T}/A$. Here, $L_{T}\leq C_{G}(A)$ and $C_{G}(T)\unlhd L_{T}$.
Since $L_{T}$ acts on $T/A$ and $A$ trivially, we have
$L_{T}/C_{G}(T)$ is a $p$-group by \cite[Appendix A, Lemma A.2]{AsKO}.
Since $(T, e_{T})$ is a $(kC_{G}(A), e_{A}, C_{G}(A))$-Brauer pair, thus
$(T, e_{T})$ is also a $(kG, b, G)$-Brauer pair  by \cite[Part IV, Lemma 3.5]{AsKO}.
That is $e_{T}$ is a block of $kC_{C_{G}(A)}(T)=kC_{G}(T)$.
Then there exists a unique block $\varepsilon_{T}$ of $kL_{T}$ such that $\varepsilon_{T}$
covers $e_{T}$ by \cite[Ch.5, Corollary 5.6]{NT}.

Set $\overline{\varepsilon}_{T}:=\mu_{A}(\varepsilon_{T})$, since
$L_{T}/C_{G}(T)$ is a $p$-group, thus $|L_{T}: C_{L_{T}}(T)|$ is power of prime. As
$C_{L_{T}}(T)\leq C_{L_{T}}(A)$,
we can see that $|L_{T}:C_{L_{T}}(A)|$ is power of prime. Hence $\overline{\varepsilon}_{T}:=\mu_{A}(\varepsilon_{T})$ is a block of $k(L_{T}/A)=kC_{C_{G}(A)/A}(\overline{T})$
by \cite[Ch. 5, Theorem 8.11]{NT}.

Since $(T, e_{T})$ is a $(kC_{G}(A), e_{A}, C_{G}(A))$-Brauer pair,
thus $\mathrm{Br}_{T}^{kC_{G}(A)}(e_{A})e_{T}\neq 0$. As
$e_{T}$ is covered by a unique $\varepsilon_{T}$, it implies that
$\mathrm{Br}_{T}^{kC_{G}(A)}(e_{A})\cdot$ $\varepsilon_{T}\neq 0$.
Since $A\unlhd C_{G}(A)$, we have $\mathrm{Ker}\mu_{A}\subseteq J(kC_{G}(A))$ by \cite[Ch.5, Lemma 2.7]{NT}.
So $\mu_{C}(\mathrm{Br}_{T}^{kC_{G}(A)}(e_{A})\varepsilon_{T})\neq 0$.
Also $\overline{e}_{A}:=\mu_{A}(e_{A})$ is a block of $kC_{G}(A)$ and\[ \xymatrix{
  (kC_{G}(A))^{T} \ar[d]_{\mu_{A}} \ar[r]^{\mathrm{Br}_{T}^{kC_{G}(A)}}
                & kC_{G}(T) \ar[d]^{\mu_{A}}  \\
  (kC_{G}(A)/A)^{\overline{T}}  \ar[r]_{\mathrm{Br}_{\overline{T}}}
                & kC_{C_{G}(A)/A}(\overline{T})          }
                \] commutes,
thus $$\mu_{A}(\mathrm{Br}_{T}^{kC_{G}(A)}(e_{A}))=\mathrm{Br}_{\overline{T}}^{kC_{G}(A)/A}(\overline{e}_{A}).$$
Hence $\mathrm{Br}_{\overline{T}}^{kC_{G}(A)/A}(\overline{e}_{A})\overline{\varepsilon}_{T}\neq 0$.
Therefore, $(\overline{T}, \overline{\varepsilon}_{T})$ is a $(kC_{G}(A)/A, \overline{e}_{A}, C_{G}(A)/A)$-Brauer pair.
So we can define $ \Phi(T, e_{T})= (\overline{T}, \overline{\varepsilon}_{T})$ and $\Phi$ is injective.

We show that the map $\Phi$ is inclusion preserving. By \cite[Part IV, Proposition 2.14]{AsKO}, we only need to prove it when
$(A, e_{A})\leq (R,e_{R})\unlhd (T, e_{T})$. So we have $R\unlhd T$, $\mathrm{Br}_{T, R}^{kC_{G}(A)}(e_{R})e_{T}=e_{T}$ and $e_{R}$ is $T$-stable.
That is $e_{R}e_{T}=e_{T}$. Hence $\varepsilon_{R}\varepsilon_{T}\neq 0$ because $\varepsilon_{R},
\varepsilon_{T}$ cover $e_{R}, e_{T}$ respectively. Now, we can see that
$0\neq \mu_{A}(\varepsilon_{R}\varepsilon_{T})= \overline{\varepsilon}_{R}\overline{\varepsilon}_{T}$.
Also let $t\in T$, $\varepsilon_{R}^{t}$ is a block of $kL_{R}$ because $T$ acts on $L_{R}$. (Here, the definition of $L_{R}$
like $L_{T}$ above.)
But there exists a unique block $\varepsilon_{R}$ of $kL_{R}$ such that $\varepsilon_{R}$
covers $e_{R}$, thus $\varepsilon_{R}^{t}=\varepsilon_{R}$ because $\varepsilon_{R}^{t}$ covers
$e_{R}^{t}=e_{R}$. (Here, $e_{R}$ is $T$-stable.) Hence
$\overline{\varepsilon}_{R}$ is $\overline{T}$-stable.
Therefore, $(\overline{R},\overline{\varepsilon}_{R})\unlhd (\overline{T}, \overline{\varepsilon}_{T})$,
that means $\Phi(R, e_{R})\unlhd \Phi(T, e_{T})$. So $\Phi$ is a poset bijection.

Let $g\in C_{G}(A)$, we have
 $\Phi(T, e_{T})^{g}=(\overline{T}, \overline{\varepsilon}_{T})^{g}
=(\overline{T}^{g}, \overline{\varepsilon}_{T}^{g})$. Since $\varepsilon_{T}$ cover
$e_{T}$, thus $\varepsilon_{T}^{g}$ covers $e_{T}^{g}$. Here,
$\mu_{A}(\varepsilon_{T})^{g}=\mu_{A}(\varepsilon_{T}^{g})$. That is
$\Phi(T, e_{T})^{g}=\Phi(T^{g}, e_{T}^{g})$. Hence $\Phi$ is compatible with
$C_{G}(A)$-action.

Now, we will prove that
$(\overline{S}, \overline{\varepsilon}_{S})$ is a maximal $(kC_{G}(A)/A, \overline{e}_{A}, C_{G}(A)/A)$-Brauer pair.

 Suppose that there exists a $(kC_{G}(A)/A, \overline{e}_{A},$ $ C_{G}(A)/A)$-Brauer pair
 $(R/A, \overline{\varepsilon}_{R})\gneq (\overline{S}, \overline{\varepsilon}_{S})$.
Hence, $R\gneq S$ and $\mathrm{Br}_{R/A}^{kC_{G}(A)/A}(\overline{e}_{A})\overline{\varepsilon}_{R}\neq 0$.
So $\mathrm{Br}_{R}^{kC_{G}(A)}(e_{A})\varepsilon_{R}\neq 0$.
There exists a block $e_{R}$ of $kC_{G}(R)$ such that $\varepsilon_{R}$ covers $e_{R}$
and $\mathrm{Br}_{R}^{kC_{G}(A)}(e_{A})e_{R}\neq 0$. So
$(R, e_{R})$ is a $(kC_{G}(A), e_{A}, C_{G}(A))$-Brauer pair.
By \cite[Part IV, Lemma 3.5]{AsKO}, we have $(R, e_{R})$ is a
$(kG, b, G)$-Brauer pair. That is a contradiction. Hence
$(\overline{S}, \overline{\varepsilon}_{S})$ is a maximal $(kC_{G}(A)/A, \overline{e}_{A}, C_{G}(A)/A)$-Brauer pair.
So we can see that $\Phi$ is surjective.

By \cite[Part IV, Proposition 3.2]{AsKO}, we have
$\mathcal{F}_{(S/A, \overline{\varepsilon}_{S})}(kC_{G}(A)/A, \overline{e}_{A}, C_{G}(A)/A)$ is a saturated fusion system
over $S/A$.

Second, let $A\leq R\leq S$, we can see that
$$N_{C_{G}(A)}(R, e_{R})/A\leq N_{C_{G}(A)/A}(\overline{R}, \overline{\varepsilon}_{R}).$$
That is $$C_{\mathcal{F}}(A)/A\subseteq \mathcal{F}_{(S/A,
\overline{\varepsilon}_{S})}(kC_{G}(A)/A, \overline{e}_{A},C_{G}(A)/A).$$

Since $C_{\mathcal{F}}(A)/A$ is a saturated fusion system over $S/A\cong p^{1+2}_{+}$,
if $C_{\mathcal{F}}(A)/A$ is not a fusion system of a $p$-block of a finite group, then
$C_{\mathcal{F}}(A)/A$ is one of the following saturated fusion systems:
\[
\begin{array}{|c|c|c|c|}\hline
\mathrm{Name}&\mbox{$\mathrm{Out}_{\mathcal{F}}(7^{1+2}_{+})$} &\mbox{$|\mathcal{F}^{cr}|$}&
\mbox{ $\mathrm{Aut}_{\mathcal{F}}(V)$}\\\hline
RV_{1}& 6^{2}: 2  & 6+2 & \mathrm{SL}_{2}(7):2, \mathrm{GL}_{2}(7) \\
RV_{2}&D_{16}\times 3 & 4+4 & \mathrm{SL}_{2}(7):2, \mathrm{SL}_{2}(7):2 \\
RV_{3}& SD_{32}\times 3  & 8  & \mathrm{SL}_{2}(7):2\\\hline
\end{array}
\]
by \cite{KeS}. And this table also can be fined from \cite[Lemma 4.17]{RV}.

Here, $C_{\mathcal{F}}(A)/A\subseteq \mathcal{F}_{(S/A,
\overline{\varepsilon}_{S})}(kC_{G}(A)/A, \overline{e}_{A},C_{G}(A)/A).$
Hence either $$C_{\mathcal{F}}(A)/A= \mathcal{F}_{(S/A,
\overline{\varepsilon}_{S})}(kC_{G}(A)/A, \overline{e}_{A},C_{G}(A)/A),$$
or $C_{\mathcal{F}}(A)/A\cong RV_{2}$ and $\mathcal{F}_{(S/A,
\overline{\varepsilon}_{S})}(kC_{G}(A)/A, \overline{e}_{A},C_{G}(A)/A)\cong RV_{3}$ by
\cite{KeS, RV}.

However, $RV_{1}, RV_{2}, RV_{3}$ are not fusion systems of a $p$-blocks of finite groups.
That is a contradiction.
Hence, $C_{\mathcal{F}}(A)/A\cong\mathcal{F}_{1}$ is a fusion system of a $p$-block of a finite group.

Conversely suppose that $\mathcal{F}_{1}$ is a fusion system of a $p$-block of a finite group.
Since $P\cong p^{1+2}_{+}$, we have that $\mathcal{F}_{1}$ is not exotic by \cite[Theorem 6.4]{KeS}.
Let $\mathcal{F}_{1}=\mathcal{F}_{P}(G_{1})$ for a finite group $G_{1}$ with $P\in \mathrm{Syl}_{p}(G_{1})$.
Set $G= G_{1}\times (A\rtimes \mathrm{Aut}_{\mathcal{F}}(A))$, then $P\times A\in \mathrm{Syl}_{p}(G)$.
So $\mathcal{F}\cong \mathcal{F}_{P\times A}(G).$ That is $\mathcal{F}$ is a fusion system of a $p$-block of a finite group.
\end{proof}



\begin{thebibliography}{13}


\bibitem{AB}  J.  Alperin, M. Brou$\mathrm{\acute{e}}$, Local methods in block theory, Ann. of Math. 110(1979) 143-157.

%
\bibitem{A5}     M. Aschbacher, Normal subsystems of fusion systems, Proc. Lond. Math. Soc. 97(2008) 239-271
\bibitem{A6}     M. Aschbacher, The generalized Fitting subsystem of a fusion system, Memoirs Amer. Math. Soc. 209 (2011), no. 986.



\bibitem{AsKO}    M. Aschbacher, R. Kessar, and B. Oliver, Fusion systems in algebra and topology, London Mathematical Society Lecture Note Series 391, Cambridge University Press, Cambridge, 2011.

\bibitem{BP}     M. Brou$\mathrm{\acute{e}}$, L. Puig, Charaters and local structures in $G$-algebra, J. Algebra 63(1980), 306-317.





\bibitem{BLO}  C.  Broto,  R.  Levi,    B.  Oliver,  The  homotopy  theory  of  fusion  systems,  J.  Amer.  Math.  Soc.  16
         (2003), 779-856.


%
%

\bibitem{Ke}    R. Kessar, The Solomon fusion system $\mathcal{F}_{\mathrm{Sol}}(3)$ does not
         occur as fusion system of a $2$-block, J. Algebra 296(2006), 409-425.

\bibitem{KeS}    R. Kessar, R. Stancu, A reduction theorem for fusion systems of blocks, J. Algebra 319(2008), 806-823.

\bibitem{LO}   R. Levi, B. Oliver, Construction of 2-local finite groups of a type studied by Solomon and Benson,
         Geometry and Topology 6 (2002), 917-990.


%

\bibitem{NT}    H. Nagao, Y. Tsushima, Representations of finite groups, Academic Press, Boston, 1988.


\bibitem{Pu}  L. Puig, Structure locale dans les groupes finis, Bul. Soc. Math. France, M\'{e}moire no. 47, PH thesis (1976).



\bibitem{RS}      K. Roberts,  S. Shpectorov,  On  the  definition  of  saturated  fusion  systems,  J.  Group  Theory  12
         (2009), 679-687.

\bibitem{RV}      A. Ruiz, A. Viruel, The classification of $p$-local finite groups over the extraspecial group of order  $p^{3}$
             and  exponent $p$, Math. Z. 248 (2004), 45-65.


%
%
\bibitem{St2}   R. Stancu, Control of fusion in fusion systems,  J. Algebra and its Applications. 5(6)
(2006), 817-837.
%

\end{thebibliography}
\end{document}